\documentclass[a4paper]{scrartcl}
\usepackage[utf8]{inputenc}
\usepackage{amsfonts,amsthm,amssymb,amsmath}
\usepackage[inline]{enumitem}
\usepackage{color}
\usepackage{graphicx} 
\usepackage{authblk}

\newtheorem{thm}{Theorem}

\theoremstyle{remark}

\definecolor{mypink1}{RGB}{219, 48, 122}

\title{Breaking small automorphisms of graphs of arbitrary cardinality}
\author{Marcin Stawiski \\ stawiski@agh.edu.pl}

\affil{AGH University of Krakow,\\ Faculty of Applied Mathematics, \protect\\al. Mickiewicza 30, 30-059 Krakow, Poland}

\begin{document}
\maketitle
\begin{abstract}

We say that an edge colouring $c$ of a graph \emph{preserves} an automorphism $\varphi$ if $\varphi$ maps each edge to an edge of the same colour. Otherwise, we say that $c$ \emph{breaks} $\varphi$. We call an automorphism of a graph \emph{small} if it moves some vertex to its neighbour. We study the edge colourings of graphs that break every small automorphism. Kalinowski, Pilśniak, and Woźniak proved that three colours are enough for such a colouring to exist for every finite graph without isolated edges. They conjectured that two colours are enough for every finite connected graph on at least six vertices. We confirm this conjecture in its more general version, namely for connected finite and infinite graphs of arbitrary cardinality.


\end{abstract}

\section{Introduction}

Babai \cite{BAB} in 1977 started investigating vertex colourings of a graph $G$ that breaks all non-trivial automorphisms of $G$. He called such a colouring \emph{asymmetric}, although today the more popular term is \emph{distinguishing}. This concept plays a vital role in his quasi-polynomial algorithm for the Graph Isomorphism Problem \cite{babaiisomorphism}. 

From now on, all the considered colourings are edge colourings. The edge variant of distinguishing colourings was introduced by Kalinowski and Pilśniak \cite{KP}. The minimum number of colours needed in a colouring which breaks every non-trivial automorphism of a graph $G$ is called its \emph{distinguishing index}, and it is denoted by $D'(G)$. Pilśniak \cite{P} characterized the finite connected graphs $G$ whose distinguishing number is equal to $\Delta(G)+1$ or $\Delta(G)$. Pilśniak and Stawiski \cite{PS} proved the optimal bound of $\Delta(G)-1$ for infinite connected graph $G$ which is not the double ray.  Kalinowski, Pil\'sniak and Wo{\'z}niak \cite{small} started the investigation of the following problem. Given any graph, we want to minimize the number of colours needed in a colouring of a given graph $G$ to break every small automorphism of $G$. They denote this number by $D'_s(G)$ and call it the \emph{small distinguishing index} of $G$. 

The problem of breaking (small) automorphisms is connected with the concept of general distinguishing (adjacent) vertices by an edge colouring. In general formulation, we have some set of properties $\Psi$ of a coloured rooted graph $(G,v,c)$, where $c$ is a colouring and $v$ is the root of $(G,v,c)$. Let $\Psi(G,v,c)$ be the set of these properties in $\Psi$ which are satisfied for $(G,v,c)$. We want to have a colouring $c$ such that for every pair of (adjacent) vertices $x$ and $y$ of $G$ the sets $\Psi(G,x,c)$ and $\Psi(G,y,c)$ are different. For example, consider only these colourings that use only natural numbers. We may want to distinguish each pair of (adjacent) vertices $x$ and $y$ of a locally finite graph by the sum of the colours on its adjacent edges. In this case $\Psi(G,v,c)=\{\Psi_n(G,v,c)\colon n \in \mathbb{N}\}$ where $\Psi_n(G,v,c)$ means that the sum of the colours on the edges adjacent to $v$ is equal to $n$. The minimum number $k$ such that the colours from the set $\{1,2,\dots\}$ suffice for the colourings that distinguish each pair of vertices is called \emph{the irregularity strength} of a graph, and it is bounded from below by the distinguishing index. This concept was introduced by Faudree,  Schelp, Jacobson, and Lehel \cite{irregular} in 1989. The local version of this problem is connected to the so-called 1-2-3 Conjecture by Karoński, \L{}uczak and Thomason \cite{KLT}, which states that for each finite graph without component isomorphic to $K_2$ each pair of adjacent vertices can be distinguished in some colouring using only colours from the set $\{1,2,3\}$. Keusch \cite{Keusch} recently proved this conjecture. Furthermore, the result of Keusch was extended to infinite graphs by Stawiski \cite{StawiskixD}. There are many examples of such general distinguishing colourings: antimagic labellings, distinguishing by the palettes, sets, multisets, or products instead of the sums. Similarly, colours required for such colourings are bounded from below by the distinguishing index.

Kalinowski, Pil\'sniak and Wo{\'z}niak \cite{small} proved that for every finite graph $G$ without isolated edges $D'_s(G)\leq 3$, and they conjectured that $D'_s(G)\leq 2$ for each connected finite graph on at least six vertices. This theorem and the conjecture were recently extended to the list version by Kwaśny and Stawiski \cite{KS:smoll}  for locally finite graphs. 

Our main theorem is the following:

\begin{thm}\label{thm:main}
    Let $G$ be a finite or infinite graph of order at least $6$. Then $D'_s(G)\leq 2$.
\end{thm}

Our proof is novel in the field of breaking automorphisms by colouring, because in approach we do not intend to fix any vertex.

Note that the confirmation of 1-2-3 Conjecture implies the result of Kalinowski, Pilśniak, and Woźniak, but not the list version of Kwaśny and Stawiski nor our main theorem.

\section{Main theorem}

In this section, we prove Theorem \ref{thm:main}. In our proof, we use the following theorems.

\begin{thm}[Kwaśny, Stawiski, 2025+ \cite{KS:regular}]
    If $G$ is a connected locally finite regular graph of order at least $7$, then $D'(G)\leq 2$.
\end{thm}

\begin{thm}[Stawiski, Wilson 2024 \cite{Wilson}]
    If $G$ is a regular graph of infinite degree, then $D'(G)\leq 2$.
\end{thm}

We say that a colouring $c$ of a graph is \emph{almost distinguishing} if there exists a pair of vertices $x$ and $y$ such that every non-trivial automorphism preserving $c$ swaps $x$ and $y$. Note that the only regular graphs which do not admit a distinguishing colouring have two non-isomorphic almost distinguishing colouring. 

Now we can proceed with the proof of Theorem \ref{thm:main}.

\begin{proof}
We can assume that $G$ is not regular. We check if there exists a vertex $x$ such that the components of its orbit $X$ (with respect to the group of automorphisms of $G$) admit a distinguishing colouring using at most $2$ colours. Otherwise, we choose a vertex $x$ in an arbitrary way. 

We colour the component $C$ of $X$ containing $x$ with a distinguishing colouring if such exists or with an almost distinguishing colouring if not, and all the other components in $X$ with a distinguishing colouring or an almost distinguishing colouring not isomorphic to the one chosen for $C$. Moreover, we colour all the components in every orbit of each remaining vertex with a distinguishing or an almost distinguishing colouring.  
We consider few possible cases as to decide what we do in the next step:

\textbf{Case I}: The partial colouring defined so far breaks every automorphism of $C$. In this case, we just proceed to the last step, which we describe later.

For all other cases assume that $C$ is not coloured with a distinguishing colouring, and there exists a vertex $y$ such that each small automorphism of $C$ preserving the current  colouring of $C$ swaps $x$ and $y$.

\textbf{Case II}: There exists a common neighbour $v$ of $x$ and $y$ outside of $X$. We colour $xv$ with pink and $yv$ with blue for one such $v$ and we make sure that in the last step, which we describe later, all the edges between $y$ and a common neighbour of $x$ and $y$ outside of $X$ get colour blue. We can proceed to the last step.

\textbf{Case III}: There is no common neighbour of $x$ and $y$ outside of $C$. Let $x'$ be some neighbour of $x$ outside of $C$ which is mapped to a neighbour $y'$ of $y$. If there are such $x'$ and $y'$ which lie in the same component of their orbit, then we colour each such edge $xx'$ with pink and each such edge $xy'$ with blue. Assume that it is not the case i.e. each such $x'$ and $y'$ lie in different components of their orbit. In this case it is enough that the component $X'$ containing $x'$ and the component $Y'$ containing $y'$ are coloured with non-isomorphic almost distinguishing colourings. In all these possible subcases $x$ and $y$ are distinguished from each other.

Now, we describe the last step, in which the remaining edges are coloured. We choose some enumeration $\mathcal{A} = (A_0, A_1, \dots)$ of orbits with respect to the group of automorphisms fixing setwise $C$ such that for every $i,j$ if $i<j$, then the vertices of $A_i$ have less or equal distance to $C$ than the vertices of $A_j$. In particular, $A_0=C$.  We refer to the edges between $a_i\in A_j$ and $a_j\in A_i$ as the \emph{back edges} of $a_i$ if $i<j$. Note that the set $\mathcal{A}$ may be of arbitrary cardinality. We proceed with a (transfinite) induction on $i$. In step $i$, we colour all the back edges of vertices in $A_i$ except for $i=1$, because they may be some edges incident to $x$ or $y$ that are already coloured. If a component $D$ of $A_i$ is distinguished by current partial colouring, then we colour back edges of vertices of $D$ in an arbitrary way. If not, then there exists a pair of vertices $a,b\in D$ such that each small automorphism of $D$ maps $a$ into $b$. We colour one of the back edges of $a$ with pink, the rest with blue, and the back edges of the remaining vertices in $C$ with blue.

Now we prove that the constructed colouring breaks every small automorphism of $G$. We argued already that all small automorphisms acting non-trivially $A_0=C$ are broken. Suppose that there exists the smallest $i\geq 1$ such that there exists a vertex $z$ in some component of $Z$ which is not fixed by the constructed colouring. However, every small automorphism of $G$ which acts non-trivially on $Z$ moves some vertex $a\in Z$ to some vertex $b \in Z$. Such automorphism cannot be preserved by the constructed colouring because $a$ has a pink back edge while $b$ has not. This completes the proof.
\end{proof}

\bibliographystyle{abbrv}
\bibliography{lit.bib}

\end{document}